\documentclass[12pt]{article}
   \usepackage{amsthm}
   \usepackage{amsmath}
   \usepackage{amssymb}
   \usepackage{amsfonts}
   \usepackage{amscd}
   \usepackage[mathscr]{eucal}
   \usepackage{verbatim}
   \usepackage{graphics}
  \usepackage{hyperref}

\newcommand{\wt}{\widetilde}

\newcommand{\zp}{\mathbb{Z}/p\mathbb{Z}}

\newcommand{\zm}{\mathbb{Z}/m\mathbb{Z}}

\newcommand{\Card}{\operatorname{Card}}

\newcommand{\card}{\operatorname{card}}
\newcommand{\la}{\langle}
\newcommand{\ra}{\rangle}
\newcommand{\hp}{\hat{+}}

\newtheorem{thm}{Theorem}
\newtheorem*{mlem*}{Main Lemma}
\newtheorem*{thm*}{Theorem}
\newtheorem{prop}[thm]{Proposition}

\newcommand{\Z}{{\mathbb Z}}
\newcommand{\Q}{{\mathbb Q}}
\newcommand{\Prob}{\mathbb{P}}
\newtheorem{lem}[thm]{Lemma}

\newtheorem*{cor*}{Corollary}
\newtheorem*{conj*}{Conjecture} 
\newtheorem{cor}[thm]{Corollary}
\newtheorem{conj}[thm]{Conjecture}
\theoremstyle{definition}

\newtheorem{rem}[thm]{Remark}

\begin{document}
\title{Number of sets with small sumset and the clique number of random Cayley graphs}
\author{Gyan Prakash
}
\date{}
\maketitle
\begin{abstract} Let $G$ be a finite abelian group of order $n$. For
any subset $B$ of $G$ with $B=-B$, the Cayley graph $G_B$ is a graph
on vertex set $G$ in which $ij$ is an edge if and only if $i-j\in B.$
It was shown by Ben  Green~\cite{Gclique} that when $G$ is a vector
space over a finite field $\zp$, then there is a Cayley graph
containing neither a complete subgraph nor an independent set of size
more than $c\log n\log\log n,$ where $c>0$ is an absolute constant. In
this article we observe that a modification of his arguments shows that
for an arbitrary finite abelian group of order $n$, there is a Cayley graph
containing neither a complete subgraph nor an independent set of size
more than $c\left(\omega^3(n)\log \omega(n) +\log n\log\log n\right)$,
where $c>0$ is an absolute constant and $\omega(n)$ denotes the number
of distinct prime divisors of~$n$.
\end{abstract} 
A graph $G =(V,E)$ consists of a finite nonempty set $V$ (vertex set) together with a prescribed set $E$ (edge set) of unordered pair of distinct elements of $V$. Each pair $x=\{u,v\} \in E$ is an edge of $G$ and $x$ is said to join $u$ and $v$ by an edge. The graph $G$ is complete if any two elements in $V$ are joined by an edge. A maximal complete subgraph of a graph is a {\em clique}
and the  {\em clique number} is the maximal order of a clique.
An independent set of a graph $G=(V,E)$ is a subset $V'$ of $V$ such
that no two points in $V'$ are connected by an edge. Given a graph $G
= (V,E)$ the complementary graph $G^c = (V',E')$ is a graph with vertex set $V'
=V$ and two elements of $V$ are joined by an edge in $G^c$ if and only if
they are not joined by an edge in $G.$ A set is an independent set in
$G$ if and only if it spans a complete subgraph in $G^c.$

\vspace{2mm}
\noindent
Ramsey proved that given any positive integer $k$,
there is a Ramsey number $R(k)$ such that any graph $G$ on $n$ vertices, with
$n \geq R(k)$, contains either a clique or an independent
set which has more than $k$ vertices. Erd\H os~\cite{Erdos-clique} showed that the Ramsey
number $R(k)$ has at least an exponential growth in $k$. 
Using a probabilistic argument, Erd\H os proved that  there exists a graph on $n$
vertices which neither contains a clique nor an independent
set of size more than $c\log n$ vertices with $c$ being a positive
absolute constant. An explicit construction of such a graph is not known. 
Chung~\cite{Chung} gave a construction of graphs on $n$ vertices which contains neither a complete subgraph nor an independent set on more than 
$e^{c(\log n)^{3/4}/(\log \log n)^{1/4}}$ vertices.

\vspace{2mm}
\noindent
Given a  finite abelian group $G$ of order $n$ and a set $B
\subset G$, with $B =-B$ and $0 \notin B$, the {\em Cayley graph} $G_B$ is a graph on
vertex set $G$ in which $ij$ is an edge if and only if $i-j\in
B$.  It is expected that for most of primes $q$ with $q \equiv 1
\mod(4)$ the Paley graphs $P_q$, which is a
Cayley graph $G_B$ with $G = \Z/q\Z$ and $B$ being a set of quadratic
residues, is an example of a graph which  contains neither a clique nor an independent set on more than $c\log n$
vertices. However this is far from being proven and is expected to
be a very difficult problem. It is easy to see that a lower
bound for clique
number of $P_q$ is $n(q)$, where $n(q)$ denotes the least positive integer which is a
quadratic nonresidue modulo $q$.  The best unconditional upper bound known for $n(q)$ is $q^{1/4\sqrt{e} + \epsilon}$ and under the assumption of generalised Riemann hypothesis one knows that $n(q)$ is at most $c\log^2 q.$ The best known upper bound for clique number of
$P_q$ to our knowledge is $\sqrt{q}$ \cite[page 363, Theorem 13.14]{Randomgraphs}. One may ask whether among Cayley graphs, there are graphs (not
necessarily Paley graphs) which
 contains neither a complete subgraph nor an independent set of very
large order. 
The following conjecture is due
to Noga Alon.
\begin{conj}\cite[Conjecture 4.1]{Cayley} There exists an absolute
constant $b$ such that the following holds. For every group $G$ on $n$ elements there
exists a set $B \subset G$ such that the
Cayley graph $G_B$ neither contains a complete subgraph nor an independent
set on more than $b \log n$ vertices.\label{cayley}
\end{conj}
\noindent
For the relation between this conjecture and certain other questions
in information theory, one may see the article of
Noga Alon~\cite{Cayley}.
A weaker version of this conjecture, obtained by
replacing the term $\log n$ by  $\log^2 n$, was proved by N. Alon and
A. Orilitsky in~\cite{alon}.

\vspace{2mm}
\noindent 
Ben Green~\cite{Gclique} proved the above conjecture in the case
when $G$ is cyclic. In the case when $G = (\Z/p\Z)^r$ with $p$ being a prime, he  proved a
weaker version of the above conjecture with the term $\log n$ replaced
by $\log{n}\log\log{n}.$
 It was shown by Green that if we select a subset $B$ of $G$
randomly, then almost surely the Cayley graph $G_B$  contains neither a
complete subgraph nor an independent set of large size.
On the other hand, Green also proved that when $G = (\Z/2\Z)^r$, then for a
random subset $B$, the Cayley graph $G_B$ almost surely contains a
complete subgraph of size at least $c \log n\log\log n$ and thus
showing that the random methods alone can not prove the above
conjecture for a general finite abelian groups. Moreover Ben Green
remarked in~\cite{Gclique} that his methods seems to work only for
certain special groups.\\

\vspace{2mm}
\noindent
In this article we observe
that a modification  of the arguments from~\cite{Gclique}
prove the following weaker version of the above conjecture for any
finite abelian group.
\begin{thm}\label{alon-conj} Let $G$ be a finite abelian group of
order $n$. Then there exist a subset $B$ of $G$ with
$B=-B$ and $0\notin B$, such that the Cayley graph $G_B$ neither
contains a complete subgraph nor an independent set on more than
$c(\omega^3(n)\log {\omega(n)} + \log{n}\log\log{n})$ vertices, where
$\omega(n)$ denotes the number of distinct prime divisors of $n$ and
$c$ is a positive absolute constant.
\end{thm}  
\noindent 
When the order $n$ of $G$ is such that $\omega(n) \leq (\log
n)^{1/3}$, then Theorem~\ref{alon-conj} gives a weaker version of
Conjecture~\ref{cayley}  with the term $\log n$ replaced by $\log \log
n.$ When $G = (\zp)^r)$, then
 $\omega(n) = 1$ and we obtain the  result of Ben Green mentioned above. 
Since sometimes $\omega(n)$ could be as large as $\frac{\log n}{\log\log n}$,
 which happens when $n$ has several small prime divisors, it is not possible to recover the result of Alon and Orilitsky from Theorem~\ref{alon-conj}. 

\vspace{2mm}
\noindent   The complementary graph of a Cayley graph $G_B$ is the
Cayley graph $G_{B^c}$ with  $B^c = G\setminus (B\cup \{0\}).$ Thus to
prove Theorem~\ref{alon-conj} we need to show the existence of set $B
\subset G$ such that the clique number of $G_B$ as well as that of
$G_{B^c}$ is small. We divide $G\setminus \{0\}$ into disjoint
pairs of the form $(g,-g)$ with $g\in G\setminus\{0\}$. Then we choose
a subset $B$ of $G$ randomly by choosing each such pair in $B$
independently with probability $1/2$. We write ${\it cl(B)}$ to denote
the clique number of the Cayley graph $G_B$. 

\vspace{2mm}
\noindent
In case $G= (\zp)^r$ with $p$ being
a prime, the following result was proved by Ben
Green~\cite[Theorem 9]{Gclique}\label{cayley-vector}, whereas we prove
it for an arbitrary  finite abelian group $G$. Green had stated and
proved his results for Cayley {\em sum} graphs and not for Cayley
graphs. However as he remarked, his arguments after a minimal
modification gives the same result for Cayley graphs. 
\begin{thm}\label{random-cayley} There exists an  absolute constant
$c_1>0$ such  that the following holds. For any finite abelian group
$G$ of order $n$ we have that
$$\lim_{n \to \infty}\Prob\left({\it cl(B)} \geq c_1(\omega^3(n)\log{\omega(n)} +
  \log{n}\log\log{n})\right) =0.$$
\end{thm}

\begin{rem}
Using the arguments of this paper and the result~\cite[Proposition 19]{Gclique}  proved by Green, one can show that the clique number of random Cayley graph is at most 
  $c_1(\omega^{\frac{3(1+\alpha)}{1+2\alpha}}(n)\log{\omega(n)} +
  (\log{n}\log\log{n})^{1+\alpha})$ for any $\alpha \in [0,1].$ When $\omega(n) \leq \log^{1/3}n$, the choice of $\alpha =0$ is optimal. Taking $\alpha =0$, we recover the result of Theorem~\ref{random-cayley}. When $\omega(n)$ is of the order $\frac{\log n}{\log \log n}$, then taking $\alpha=1$, we obtain the bound $c_1(\log n\log\log n)^2.$
\end{rem}

\vspace{3mm}
\noindent We observe that Theorem~\ref{alon-conj} follows immediately
from Theorem~\ref{random-cayley}, using the following inequality:
$$\Prob({\it cl(B)}\geq k_1 \text{ or } {\it cl(B^c)}\geq k_1)
\leq \Prob({\it cl(B)} \geq k_1) + \Prob({\it cl(B^c)}\geq k_1)
 = 2\Prob({\it cl(B)} \geq k_1),$$
where the last equality follows using the fact that for any pair $\{g,-g\}$ with $g \in G\setminus \{0\}$, the probability that the pair belongs to $B$ is equal to the probability that it belongs to $B^c.$

\vspace{2mm}
\noindent
For any positive integers $k_1$ and $k_2$ we set
\begin{eqnarray}
S^-(k_1,k_2,G) &=& \{A \subset G: \card(A)= k_1, \card(A-A) = k_2\},\label{diffset}
\end{eqnarray}
where
 $A -A$ denotes the subset of $G$ consisting of those elements
which can be written as a difference of two elements from $A$. 
In~\cite{Gclique}, Green observed the following inequality which relates
the clique number of random Cayley graph and the cardinality of $S^-(k_1,k_2,G)$.
\begin{equation}\label{relc} \Prob({\it cl(B)}\geq k_1) \leq \sum_{k_2 \geq k_1}\frac{\Card\left(S^-(k_1,k_2,G)\right)}{2^{(k_2-1)/2}}.
\end{equation} 
Presently, we recall the arguments from~\cite{Gclique} which prove~\eqref{relc}.
 The probability that the clique number ${\it cl(B)}$ of a
random Cayley graph $G_B$ is greater than or equal to $k_1$ is same as
the probability that there exist  a set $A \subset G$ with $\card(A)
=k_1$ which spans a complete subgraph in $G_B$. The subgraph  of $G_B
$ spanned by the vertices of $A$ is  complete if and only if
$(A-A)\setminus \{0\}$ is a subset of $B$. If $\card(A-A)= k_2$, it
contains at least $\frac{k_2-1}{2}$ disjoint pairs of the form
$(g,-g)$ with $g\in G\setminus \{0\}$. Thus the probability that $A$
spans a complete subgraph is at most $\frac{1}{2^{(k_2-1)/2}}$. Therefore we have
\begin{equation*} \Prob({\it cl(B)}\geq k_1) \leq \sum_{k_2
\geq k_1}\sum_{A \in S^-(k_1,k_2,G)}\Prob((A-A)\setminus \{0\} \subset
B) \leq \sum_{k_2 \geq k_1}\frac{\Card\left(S^-(k_1,k_2,G)\right)}{2^{(k_2-1)/2}}.
\end{equation*} 
For any positive integers $k_1$ and $k_2$ we also set 
\begin{equation}
S(k_1,k_2,G) = \{A \subset G: \card(A) =k_1, \card(A\hp A) \leq k_2\}, \label{sumset}
\end{equation}
where $A\hp A$ denotes those elements of $G$ which can be written as a
sum of two distinct elements of $A.$

\vspace{2mm}
\noindent
 The following result was stated in~\cite{Gclique} when $G = (\zp)^r$ with $p=2$,
but the arguments give the same result when $p$ is an arbitrary
prime. Moreover the arguments gives the same upper bound for $\card(S^-(k_1,k_2,(\zp)^r).$
\begin{thm}\cite[Proposition 26]{Gclique} For any prime  $p$, we
  have, \label{smallvector}
\[ \Card(S(k_1,k_2,(\Z/p\Z)^r)) \leq n^{\frac{4k_2\log
k_1}{k_1}}\left(\frac{ek_2}{k_1}\right)^{k_1} \exp(k_1^{31/32})
\] if $k_2 \leq k^{31/30}$ and
\[ \Card(S(k_1,k_2,(\Z/p\Z)^r)) \leq n^{\frac{4k_2\log
k_1}{k_1}}k_1^{4k_1}
\] for all $k_2$. (Here $n = p^r$ is the order of $(\Z/p\Z)^r.$)
\end{thm} 
We prove the following result.
\begin{thm}\label{smallG} Let $G$ be a finite abelian group of order
$n$. Then the cardinality of $S^-(k_1,k_2,G)$ as well
as the cardinality of $S(k_1,k_2,G)$ is at most
\begin{equation}\label{Gbound}
n^{\frac{4k_2\log k_1}{k_1}}\min(k_1^{c\omega(n)(k_1k_2)^{1/3}\log
  k_1}\binom{k_2}{k_1-1}(k_1^3 +1), k_1^{4\omega(n)k_1}),
\end{equation} where $c$ is a positive absolute constant.
\end{thm}
\noindent To prove Theorem~\ref{smallvector}, Green proved the
following:
\begin{enumerate}
\item an upper bound for the number of Freiman 2-isomorphism class of
  sets in $S(k_1,k_2,G)$, 
\item an upper
  bound for the cardinality of the set $Hom_2(A,G)$, where $Hom_2(A,G)$ consists of all
  Freiman homomorphism from $A$ into $G,$
\end{enumerate}
when $G = (\zp)^r.$ 
We prove Theorem~\ref{smallG} by proving the same for general $G.$
For obtaining an upper bound for $\card(Hom_2(A,G))$, 
 we observe that
$A$ is Freiman 2-isomorphic to a subset $A_{r,2}$ of a possibly different group $G'$ such that $A_{r,2}$ have the following ``universal'' property. Any Freiman $2$-homomorphism from $A_{r,2}$ into $G$ extends as a group homomorphism from the group $\la A_{r,2}\ra$ into $G$, where $\la A_{r,2}\ra$ is the subgroup of $G'$ generated by $A_{r,2}$. Hence the group $Hom_2(A_{r,2},G)$ is isomorphic to $Hom(\la A_{r,2} \ra, G)$ (Lemma~\ref{extension}), where $Hom(\la A_{r,2}\ra, G)$ is the group consisting of all group homomorphism from $\la A_{r,2}\ra$ into $G$. This shows that $\card(Hom_2(A,G)) \leq n^{r(\la A_{r,2}\ra)}$,
 where $r(\la A_{r,2}\ra)$ is the rank of the group $\la
 A_{r,2}\ra$. An upper bound for the rank of $\la A_{r,2}\ra$ follows
 from a result proved by Green. The arguments used by Green in obtaining an upper bound for the number
of Freiman 2-isomorphism classes of sets works for general $G$ without
much difficulty. We need to use Lemma~\ref{gf} which follows from a
standard inductive argument.

\vspace{2mm}
\noindent
Given a positive integer $s$, for any finite subset $A$ of an $F$-module with $F$ being one of the following two rings $\zm$ and $\Q$, in Section~\ref{Freiman} we define the Freiman $s$-rank $r_s(A)$ to be the rank of the module $Hom_s(A,F)$. We prove 
Corollary~\ref{analogous} which generalises the result~\cite[Corollary 14]{Gclique} proved in the case of $F$ being a field. Although we do not require Corollary~\ref{analogous} to prove other results of this article, the result may be of an independent interest. The result shows that in case $F= \Q$, the Freiman $2$-rank of $A$ as defined above is same as the rank of $A$ as defined by Freiman. Using this fact Green observed that  the factor $n^{\frac{4k_2\log k_2}{k_1}}$ in~\eqref{Gbound} could be improved to $n^{\frac{4k_2}{k_1}}$ for a cyclic group, which allowed him to prove Conjecture~\ref{cayley} for cyclic groups.
\section{Number of sets with small sumset}
Let $m$ be a fixed positive integer. In the sequel, we fix $F$ to be either $\zm$ or $\Q$. Let $M$ be a finitely generated $F$-module.
If $F = \zm$, then  $M$ is a finite abelian group of  exponent $m'$ which is a divisor of $m$ and in case $F = \Q$ then $M$  is a
finite dimensional vector space over $\Q$. Given any subset $A$ of $M$ we write $\la A\ra$ to denote the submodule of $M$
spanned by $A$. Notice that if $F =\zm$, then $\la A\ra$ is same as the
subgroup
generated by $A$, but if $F = \Q$ then in general the subgroup generated 
by $A$  is a proper subset of
 $\la A\ra$.
Given any finite subset $C$ of
$M$, we  set
 $$S(k_1,k_2,C,M) = \{A \in S(k_1,k_2,M): A \subset C\},$$

$$\text{ and } S^-(k_1,k_2,C,M)) = \{A \in S^-(k_1,k_2,M): A \subset C\},$$
where $S(k_1,k_2,M)$ and $S^-(k_1,k_2,M)$ are as defined in~\eqref{sumset}
 and~\eqref{diffset} respectively. 

\vspace{1mm}
\noindent
For the purpose of obtaining an upper bound for clique number of
random Cayley sum graphs in a cyclic group of order $n$, an
upper bound for the cardinality of $S(k_1,k_2,C,M)$ with $M = F= \Q$ and 
$C =\{0,1,\ldots,n-1\}$ was used by Green in~\cite{Gclique}.

\vspace{2mm}
\noindent {\em Freiman s-homomorphism}: Let $s$ be a positive integer,
let $A$ and $B$ be subsets of (possibly different) abelian groups and
let $\phi: A \to B$ be a map. Then we say that $\phi$ is a Freiman
$s$-homomorphism if whenever $a_1,\ldots,a_s,a_1'\ldots,a_s' \in A$
satisfy
\begin{equation}\label{rel-domain} a_1 + a_2 +\ldots +a_s = a_1' +
a_2' +\ldots + a_s'
\end{equation} we have
\begin{equation}\label{rel-range} \phi(a_1) + \phi(a_2) +\ldots
+\phi(a_s) = \phi(a_1') + \phi(a_2') +\ldots + \phi(a_s').
\end{equation} If $\phi$ has an inverse which is also $s$-homomorphism
then we say that it is a Freiman $s$-isomorphism. We shall refer to Freiman
$2$-homomorphisms simply as Freiman homomorphisms.

\vspace{2mm}
\noindent We shall obtain an upper bound for $\card(S(k_1,k_2, C,M))$ by
obtaining an upper bound for the number $c(k_1,k_2,C,M)$ of Freiman isomorphism
classes of sets in $S(k_1, k_2, C,M)$ and an upper bound for the number
$n(A,C)$ of subsets of $C$ which are Freiman isomorphic to $A$ for any given 
$A \in S(k_1,k_2,C, G)$. Then we have 
\begin{equation}\label{strategy}
\Card\left(S(k_1,k_2,C,M)\right)\leq c(k_1,k_2,C,M)
\max_{A\in S(k_1,k_2,C,M)}n(A,C).
\end{equation}
Using similar arguments we shall obtain an upper bound for 
$\Card\left(S^-(k_1,k_2,C,M)\right)$.

\vspace{2mm}
\noindent Let $A$ be a subset of $M$ with $\card(A) = k_1$.  Let $e_1,
e_2, \ldots, e_{k_1}$ be the canonical basis of $F^{k_1}$. We write
$R_s$ to denote the subset of $F^{k_1}$ consisting of the elements of
the form
\[ e_{i_1} + e_{i_2} + \ldots +e_{i_s} - e_{j_1} -e_{j_2} - \ldots -
e_{j_s},
\] where $i'$s and $j'$s need not be distinct. For any subset $A =
\{a_1, a_2. \dots, a_{k_1}\} \subset G$, let $\phi : F^{k_1} \to G$ be
the $F$-linear map with $\phi(e_i) = a_i$. We write $R_s(A)$ to denote
the set $R_s \cap \ker(\phi)$. Let 
$A_{r,s} = \{\bar{e_{1}}, \ldots,\bar{e_{k_1}}\}$ be the image of
$\{e_{1},e_{2},\ldots, e_{k_1}\}$ in $F^{k_1}/\la R_s(A)\ra$
under the natural projection map from $F^{k_1}$ to $F^{k_1}/\la
R_s(A)\ra$. Then $\phi$ induces a map $\bar{\phi}: A_{r,s} \to A$.
\begin{lem}\label{Ars} With the notations as above, the map
$\bar{\phi}: A_{r,s} \to A$ is a Freiman $s$-isomorphism.
\end{lem}
\begin{proof} Since $\bar{\phi}$ is a restriction of group
homomorphism, it follows that it is a Freiman
$s$-homomorphism. Moreover it is evident that $\bar{\phi}$ is a
bijective map. To prove that $\bar{\phi}$ is a Freiman
$s$-isomorphism we need to show that
\begin{equation}\label{image} \bar{\phi}(\bar{e_{i_1}}) + \ldots +
\bar{\phi}(\bar{e_{i_s}}) -  \bar{\phi}\bar{(e_{j_1}}) -
\ldots - \bar{\phi}(\bar{e_{j_s}}) = 0
\end{equation} implies that
\begin{equation}\label{domain}\bar{ e_{i_1}} + \ldots + \bar{e_{i_s}} -  \bar{e_{j_1}} - \ldots
-\bar{e_{j_s}} = 0.
\end{equation} 
From~\eqref{image}, it follows that $ e_{i_1} + \ldots + e_{i_s} -  e_{j_1} - \ldots
-e_{j_s} \in \ker(\phi) \cap R_s = R_s(A).$ Therefore it follows that~\eqref{domain} holds. Hence the lemma follows.
\end{proof}
\subsection{Number of sets in a given Freiman $2$-isomorphism class}
Given any $F$-modules $H$, $H'$ and a subset $B$ of $H'$, we write $Hom_s(B, H)$ to denote the space of
Freiman $s$-isomorphism from $B$ into $H$.  We also write
$Hom_F(\langle B \rangle, H)$ to denote the space of $F$-linear map
from $\la B\ra$ into $H$.  Notice that $Hom_s(B, H)$ and
$Hom_F(\la B\ra, H)$ are $F$-modules.
\begin{lem}\label{extension} Let $H$ be a $F$ module. Then any $g \in
Hom_s(A_{r,s}, H)$ extends as a $F$-linear map $\wt{g}: \la A_{r,s}\ra
\to H$. The map thus obtained from $Hom_s(A_{r,s}, H)$ to $Hom_F(\la
A_{r,s}\ra, H)$ is an isomorphism of modules.
\end{lem}
\begin{proof} Let $g \in Hom_s(A_{r,s}, H)$. Since $F^{k_1}$ is a free
module and $e_i$'s are canonical basis of $F^{k_1}$ we have the
following $F$-linear map $g': F^{k_1} \to H$ with $g'(e_1) =
g(\bar{e_1})$. Let $x \in R_s(A)$, then $x =
e_{i_1}+e_{i_2}+\ldots+e_{i_s} -
e_{j_1}-e_{j_2}-\ldots,-e_{j_s}$. Then from the definition of $g'$ and
the fact that $g$ is a Freiman $s$-homomorphism, it follows that
$R_s(A) \subset \ker(g')$, implying that $\la R_s(A)\ra \subset
\ker(g')$. Therefore we have the $F$-linear map $\wt{g}:F^{k_1}/\la
R_s(A)\ra \to H$ with $\wt{g}(\bar{e_i}) = g(\bar{e_i})$. Since $\la
A_{r,s}\ra = F^{k_1}/\la R_s(A)\ra$, the map $\wt{g}$ is an extension
of $g$. Therefore we have a $F$-linear map $f: Hom_s(A_{r,s},H) \to
Hom_F(\la A_{r,s}\ra,H)$ with $f(g) = \wt{g}$ for any $g \in
Hom_s(A_{r,s}, H)$. It is evident that $f$ is injective. Moreover $f$
is surjective, since the restriction of any map in $Hom_F(\la
A_{r,s}\ra, H)$ to $A_{r,s}$ is a Freiman $s$-homomorphism. Thus $f$
is an isomorphism of modules.
\end{proof} 
\begin{lem}\cite[Lemma 25]{Gclique}\label{bgen} 
Let $H$ be a $F$-module. Then for any finite subset $B$ of $H$, there exists a subset
$X$ of $B$  with $\card(X) \leq \frac{4k_2 \log{k_1}}{k_1}$,
where $k_1 = \card(B)$ and $k_2$ is equal to $ \min\left(\card(B \hp B), \card(B-B)\right)$,
such that $\la X\ra = \la B\ra.$ 
\end{lem}
\begin{proof}
For any positive integer $l$, let $lB$ denotes the subset of $H$ consisting of
those elements which can be written as a sum of $l$ elements of $H$.
Since $\card(B + B) \leq \card(B\hp B) + \card(B)$, 
using Pl\"{u}nnecke-Ruzsa inequality, we verify that for any positive integer $l$, we have $$\card(lB) \leq 
\left(\frac{k_2+k_1}{k_1}\right)^l.$$
Let $\prec$ be an arbitrary ordering on $H$.
Choose  a subset $X$ of $B$ with the property that the sums $x_1+x_2+\cdots+x_l
(x_1\prec x_2 \prec\cdots \prec x_l)$ are all distinct, with 
$l = [\log_e{k_1}]$, and which is maximal 
with respect to this property. It follows from the definition of $X$ that $B
\subset hX -(h-1)X$ and thus $\la X\ra =\la B\ra$.
Moreover from the definition of $X$ we also have
 $\binom{\card(X)}{l}$ is at most
$\card(lB)$. Using this we verify that $\card(X)
\leq \frac{4k_2\log k_1}{k_1}$. Hence the lemma follows.
\end{proof}

\begin{prop}\label{1gen}
Let $M$ be a $F$-module and $C$ is a finite subset of $M$. For any finite subset $A$  of $M$, 
the number of subsets of $C$ which are Freiman $2$-isomorphic to $A$ is at
most $\card(C)^{\frac{4k_2 \log{k_1}}{k_1}}$, where $k_1$ is equal to $\card(A)$ and
$k_2$ is equal to $\min\left(\card(A \hp A), \card(A-A)\right).$
\end{prop}
\begin{proof}
The number of subsets of $C$ which are Freiman $2$-isomorphic to $A$ is at
most the number of $g$ in $Hom_2(A, \la C\ra)$ with $g(A) \subset C$. Since $A$ and
$A_{r,2}$ are Freiman $2$-isomorphic, this number is at most the number of
$g'$ in $Hom_2(A_{r,2}, \la C\ra)$ with $g'(A_{r,2}) \subset C$. Using
Lemma~\ref{extension}, this is at most the number of $F$-linear map $\wt{g}$
in $Hom_F(\la A_{r,2}\ra,\la C\ra)$ with $\wt{g}(A_{r,2}) \subset C$. Using Lemma~\ref{bgen}, we
have that the module $\la A_{r,2}\ra$ is spanned by a subset $X$ of $A_{r,2}$ with
$\card(X) \leq \frac{4k_2 \log{k_1}}{k_1}$. Since $\wt{g}$
is uniquely determined by its value on $X$, the number of such $\wt{g}$ is
at most $\card(C)^{\frac{4k_2 \log{k_1}}{k_1}}$. Hence the proposition follows.
\end{proof}
\subsection{Number of Freiman isomorphism classes}\label{classes}
We set $g(F)$ to be equal to $1$ in case $F$ is a field and to be equal to the
number of distinct prime divisors of $m$, when $F=\zm$.
We shall need the following lemma.
\begin{lem}\label{gf}
For any subset $R$ of $F^k$, there exists a subset $R_0$ of $R$ with
$\card(R_0)\leq g(F)k$ such that $\la R_0\ra=\la R\ra.$
\end{lem}
\begin{proof}
When $F$ is a field, the dimension of the subspace $\la R\ra$ of $F^k$
is at most $k$ and there exists a
subset $R_0$ of $R$ which forms a basis of the vector space $\la R\ra$. Thus the
lemma follows in this case.

\vspace{2mm}
\noindent
 Now we need to prove the lemma in case when 
$F=\zm$. In this case we shall prove the lemma by an induction on $k$.

\vspace{2mm}
\noindent
 We
first prove the lemma in case $k=1$. In this case $\la R\ra$ is equal to a
subgroup of $\zm$. Let $p:\Z \to \zm$ be the natural projection map and for
any $x\in\zm$, we write $\wt{x}$ to denote the integer in $[0,m-1]$ with $p(\wt{x})=x$.
 
\vspace{2mm}
\noindent
If the order of $\la R\ra$ is $d$, then $p^{-1}(\la R\ra) = \frac{m}{d}\Z.
$ Thus for any prime divisor $p$ of $m$,
 there exists $r_p \in R$ such that $\wt{r_p} = \frac{m}{d}\wt{r'_p}$ with $p$ not
 dividing $\wt{r'_p}$. Let $R_0 =\{r_p\}_{p|m}$. We claim that $\la R_0\ra =
 \la R\ra$. 

\vspace{2mm}
\noindent
Suppose the claim is not true. Then $\la R_0\ra$ is a proper subgroup of
 $\la R\ra$ and there exists a positive integer $d'$ which divides $m$ such that
 $p^{-1}(\la R_0\ra)$ consists of those integers which are divisible by
 $\frac{m}{d}d'$. But by construction of $R_0$ we have that for any prime
 $p|d'$ we verify that $\wt{r_p}$ is not divisible by $\frac{m}{d}d'$. This
 contradiction proves the claim and $\la R_0\ra = \la R\ra$. 
Moreover by the construction of the $R_0$, we have $\card(R_0) \leq \omega(m).$
Hence the lemma follows in case $k=1$.

\vspace{2mm}
\noindent
Now suppose the lemma is true for any $k \leq l-1$
 with $l\geq 2$. We shall show that the lemma holds for $k=l$. Let
 $\pi_1:F^l\to F$ be the projection map on the first co-ordinate. Then
 $\pi_1(\la R\ra)$ is the module of $F$ and using the fact that the lemma holds for $k=1$, it
 follows that there exist $R'_0 \subset R$ with $\card(R'_0) \leq g(F)$ such that 
$\pi_1(\la R'_0\ra) = \pi_1(\la R\ra)$. Thus for any $r\in R$, there exist $r_1 \in
\la R'_0\ra$ such that $\pi(r-r_1)=0$. Let $R'' = \{r-r_1: r\in R\}$. Then $R''
\subset F^{l-1}$ and by the induction hypothesis there exist a subset $R''_0$
of $R''$ such that $\card(R''_0) \leq g(F)(k-1)$ and $\la R''\ra = \la R''_0\ra$.
Let $R_0 =R'_0\cup R''_0$. Since $\la R\ra = \la R''\ra + \la R'_0\ra$, it follows
that $\la R_0\ra = \la R\ra $. 
Moreover we have that $\card(R_0) \leq \card(R'_0) + \card(R''_0) \leq
g(F)k$. Hence the lemma follows.
\end{proof}
The following lemma is a generalisation of~\cite[Lemma 11]{Gclique}.
\begin{lem}\label{dependence}
Let $H$ be an $F$-module. Then the number of Freiman $s$-isomorphism classes
of subsets of $H$ of the
cardinality $k$ is at most $k^{2sg(F)k}$.
\end{lem}
\begin{proof}
Let $c(k)$ be the number of Freiman $s$-isomorphism classes of subsets of $H$
  of the
cardinality $k$.
From Lemma~\ref{Ars}, any subset $B$ of the cardinality $k$ is isomorphic to
  $B_{r,s}$, which is the image of canonical basis of $F^k$ under the
  projection map from $F^{k}$ to $F^{k}/\la R_s(B)\ra$ where $R_s(B)$ is a
  subset of $R$. Thus $c(k)$ is at most the number of submodules
of $F^k$ which are spanned by a subset of
$R_s$. Using Lemma~\ref{gf} any such submodule is spanned by a subset $R_0$ of
$R_s$ of cardinality at most $g(F)k$. Thus $c(k) \leq
  \sum_{i=0}^{g(F)k}\binom{k^{2s}}{i} \leq k^{2sg(F)k}$. 
\end{proof}
Using Lemma~\ref{Ars} the Freiman $s$-isomorphism class of any subset $A$ of
an $F$-module $H$ is determined by
$s$-relation satisfied by it. Using this  and  the arguments used in the proof of~\cite[Lemma 16]{Gclique},
we obtain the following result.
\begin{lem}\cite[Lemma 16]{Gclique} \label{extra}
Let $H$ be an $F$-module. Fix a non-negative integer $t$ and a subset $B$ of
$M$ with $\card(B)=l$. Then the number of mutually non-isomorphic sets $A$
with
$\card(A) = l+t$, such that there exists a subset $A_0\subset A$ satisfying
$A_0$ is Freiman $3$-isomorphic to $B$ is at most $(l^3 +1)^{t^4}$.
\end{lem}
For any subset $A$ of an $F$-module $H$, let $A_0$ be a subset of $A$ of the
minimum possible cardinality among the subsets of $A$ satisfying the property
 that there exists  $a^* \in A$ such that $a^* + (A\setminus \{a^*\})
 \subset A_0 \hp A$. Among all the possible choices of $A_0$, we choose the one
with the minimum possible cardinality of $A_0 \hp A_0$. 
 For any positive
integers $s_1, s_2$, we define the following subset of $S(k_1,k_2,C,M)$.
\begin{equation}
S(k_1,k_2,s_1,s_2,C,M) = \{A \in S(k_1,k_2,C,M): \card(A_0) = s_1,
\card(A_0 \hp A_0) = s_2\}.\label{intros1s2}
\end{equation} 
For any $A \in S^-(k_1,k_2,C,M)$, we also choose a subset $A_0$ of $A$ which
is of the minimum possible cardinality among the subsets of $A$,
satisfying that there exist an $a^* \in A$ such that $a^* -A \subset
A_0-A_0$. Among all the possible choices of $A_0$ we choose a one with the
 cardinality of $A_0 -A_0$  minimal possible. For any positive integers $s_1$
 and $s_2$ we set
\[
S^-(k_1,k_2,s_1,s_2,C,M) = \{ A \in S^-(k_1,k_2,C,M): \card(A_0) = s_1,
\card(A_0 - A_0) = s_2\}.
\]
The following lemma is an easy exercise.
\begin{lem}\cite[Lemma 16]{Gclique}\label{descent}
Suppose that $X \cong_6 X'$. Then $X\hp X \cong_3 X'\hp X'$ and any subset
$B \subset X\hp X$ is $3$-isomorphic to a subset of $X'\hp X'$. Similarly
$X - X \cong_3 X'-X'$ and any subset $B$  of $X-X$ is Freiman $3$-isomorphic
to a subset of $X'-X$.
\end{lem}
Using Lemmas~\ref{dependence}, \ref{extra},  \ref{descent} and the argument used in the proof
of~\cite[Proposition 18]{Gclique} we obtain the following result.
\begin{prop}\label{nfc}
Let $M$ be an $F$-module. Then the number of Freiman $2$-isomorphism classes
of sets in $S(k_1,k_2,s_1,s_2,C,M)$ as well as in $S^-(k_1,k_2,s_1,s_2,C,M)$ is at most 
$(s_1)^{12g(F)s_1}\binom{s_2}{k_1-1}(k_1^3+1)$.
\end{prop}
\noindent
Now we obtain an upper bound for the cardinality of $A_0$ for any $A\in S(k_1,k_2,C,M)$.
\begin{lem}
For any $A \in S(k_1,k_2,C,M)$, there exist $a^* \in A$, $A'_0 \subset A$ and $A_1 \subset A$
with $\card(A'_0) + \card(A\setminus A_1)\ll
(k_1k_2\log{k_1})^{1/3}$ such that $a^* + A_1\subset A'_0\hp A'_0$. Similarly
for any $A \in S^-(k_1,k_2,C,M)$, there exist $a^* \in A$, $A'_0 \subset A$ and $A_1 \subset A$
with $\card(A'_0) + \card(A\setminus A_1)\ll
(k_1k_2\log{k_1})^{1/3}$ such that $a^* - A_1\subset A'_0- A'_0$.
\end{lem}
\begin{proof}
The proof follows from the arguments used in the proof
of~\cite[Proposition 15]{Gclique} with the choice of the parameters $Q$ to be 
$[\frac{k_1^{4/3}}{k_2^{2/3}}\log^{1/3}k_1]$ and $q$ to be $100\frac{\ln^{1/2}
  k_1}{\sqrt{Q}}$. In~\cite[Proposition 15]{Gclique} it was assumed that $k_2
\leq k_1^{31/30}$ and the choice of parameters $Q$ and $q$ used
 were $[k_1^{1/5}]$ and $k_1^{-1/15}$
respectively .
\end{proof}
\begin{cor}\label{sa0}
For any $A \in S(k_1,k_2, C,M)$, let $A_0$ be a subset of $A$ as define 
above. Then  we have $\card(A_0)\ll
(k_1k_2\log {k_1})^{1/3}$. Similar statement holds for
any
$A \in S^-(k_1,k_2,C,M)$.
\end{cor}
\begin{proof}
For any $A\in S(k_1,k_2,C,M)$, let $A_1, A'_0$ be subsets of $A$ as provided 
by the previous
lemma. We take $A''_0 = A'_0 \cup \{a^*\} \cup (A \setminus A_1)$. Then it follows
that $a^* + (A \setminus \{a^*\}) \subset A''_0 \hat{+} A''_0$ and 
$\card(A''_0) \ll (k_1k_2\log {k_1})^{1/3}$. This proves the claim for any 
$A\in S(k_1,k_2,C,M)$. Similar arguments prove the claim for 
any $A \in S^-(k_1,k_2,C,M)$
\end{proof}
\section{Proof of Theorems~\ref{smallG} and \ref{random-cayley}}\label{pfsmallG}
\begin{proof}[Proof of Theorem~\ref{smallG}]
 Using Proposition~\ref{nfc}, Lemmas~\ref{sa0} and~\ref{dependence} with $F = \zm$ and $M=C=G$, it follows that there exist an absolute constant $c >0$ such that 
the number of Freiman
 isomorphism classes of sets in $S(k_1,k_2,G)$ is at most
$$\min\left(k_1^{c\omega(n)(k_1k_2\log {k_1})^{1/3}}\binom{k_2}{k_1-1}(k_1^3+1),k_1^{4k_1}\right).$$
For obtaining the above estimate we have also used the fact that
$\card(A_0 \hp A_0) \leq k_2$ and since $m$ is the exponent of $G$, we have $\omega(m)=\omega(n).$ Similar arguments shows that the same upper
bound holds for the number of Freiman isomorphism classes of sets in
$S^-(k_1,k_2,G).$
Then the theorem follows using~\eqref{strategy} and Proposition~\ref{1gen} with $C = M=G$.
\end{proof}
\begin{proof}[Proof of Theorem~\ref{random-cayley}]
For any $A \in S^-(k_1,k_2,G)$, let $A_0$ be a subset of $A$ as defined above.
Since $a^* -A \subset A_0 - A_0$, we have $\card(A_0 -A_0) \geq k_1$. 
Moreover from Lemma~\ref{sa0} we have that $\card(A_0) \ll
(k_1k_2\log{k_1})^{1/3}$. Thus if $k_1$ is sufficiently large, then there
exists a subset  $A'$ of $G$ with $A_0 \subset A' \subset A$ such that
we  have $\card(A') \geq \frac{k_1}{100}$ and $\card(A'-A') \geq
100\card(A')$.
Now if $A$ spans a complete subgraph in a random Cayley graph $G_B$ then so
does $A'$. Therefore we obtain
\begin{equation}\label{vrelc}
  \Prob({\it cl(B)} \geq k_1) \leq \sum_{k_1/100 \leq k_1' \leq k_1, k_2' \geq
  100k_1'}\frac{\card(S(k_1',k_2',G))}{2^{(k_2'-1)/2}}.
\end{equation}
 Then using Theorem~\ref{smallG} we verify the
  following inequality.
\begin{equation}\label{refrelc}
\Prob({\it cl(B)} \geq k_1)\leq \sum_{k_1/100 \leq k_1' \leq k_1, k_2' \geq
  100k_1'}2^{-k_2'g(k_1',k_2',n)},
\end{equation}
with $$g(k_1',k_2',n) =
 - \frac{c\omega(n)(k_1'\log{k_1'})^{1/3}\log{k_1'}}{k_2'^{2/3}}
  -\frac{1}{k_2'}\log{\binom{k_2'}{k_1'-1}} - \frac{4\log{k_1'}\log{n}}{k_1'}+ 1/2 -\frac{1}{2k_2'}.$$
Since $k_2' \geq 100k_1'$, using the inequality $\binom{k_2'}{k_1'} \leq
  \left(\frac{ek_2'}{k_1'}\right)^{k_1'}$, it follows that there exist an 
absolute constant
  $c_1$ such that for $k_1' \geq c_1\left(\omega^3(n)\log{\omega(n)} +
  \log{n}\log\log{n}\right)$, then $g(k_1',k_2',n) \geq c_2$, for some absolute
  constant $c_2>0$. Using this and~\eqref{refrelc}, the theorem follows.
\end{proof}
\section{Freiman rank of a set}\label{Freiman}
In this section we prove 
Corollary~\ref{analogous} which was proven by Ben Green in~\cite[Corollary 14]{Gclique} in
the case when $F$ is a field. Although the result is not required for
proving other results of this article, it may be of an independent interest.

\vspace{2mm}
\noindent
{\em Rank of an $F$-module:} For any $F$-module $H$, the rank of $H$ is the
least non negative integer $r(H)$ such that there is a surjective $F$-linear
map from $F^{r(H)}$ to $H$.

\vspace{2mm}
\noindent
{\em Freiman $s$-rank}: Given any finite subset $B$ of a $F$ module
$H$ and a positive integer $s$,
we define
Freiman $s$-rank $r_s(B)$ to be $r\left(Hom_s(B,F)\right) -1$. In case $F$ is a
field and $s=2$, $r_s(B)$ is the Freiman dimension of $B$ as defined by 
Ben Green in~\cite{Gclique}.

\vspace{2mm}
\noindent
We will need the following well known fact.
\begin{lem}\label{dual}
Let $F$ be either equal to $\zm$ or is equal to $\Q$. For any finitely
generated $F$-module $H$,
the dual module $Hom_F(H, F)$ is isomorphic to $H$. 
\end{lem}
\begin{lem}\label{rsetmod}
$r_s(A) = r_s(A_{r,s}) = r(\la A_{r,s}\ra) -1$.
\end{lem}
\begin{proof}
Since $A$ and $A_{r,s}$ are Freiman $s$-isomorphic, the first equality
follows.
From Lemma~\ref{extension} the module $Hom_s(A_{r,s},F)$ is isomorphic to the
module $Hom_F(\la A_{r,s}\ra,F)$, which from Lemma~\ref{dual} is isomorphic to
$\la A_{r,s}\ra$. Hence the second equality follows.
\end{proof}
\begin{lem}\label{exponent}
There exists a unique $F$-linear map $\phi_0:\la A_{r,s}\ra\to F$ with $\phi_0(x)
=1_F$ for any $x \in A_{r,s}$. 
In case $F = \zm$, and hence $\la A_{r,s}\ra$ is a finite abelian group, the
order of any element in $A_{r,s}$ is equal to 
$m$.
\end{lem}
\begin{proof}
The constant map  $\phi'_0:A_{r,s} \to F$ with $\phi'_0(x) = 1_F$ for any
$x\in A_{r,s}$ is a Freiman $s$-homomorphism. Therefore using
Lemma~\ref{extension}, there exists a unique $F$-linear map $\phi_0:\la A_{r,s}\ra\to F$
with $\phi_0(x) = 1_F$ for any $x\in A_{r,s}$. This proves the first part of
the lemma. In case $F= \zm$, let $x$ be
any fixed element in $A_{r,s}$ and $d$ be the order of $x$. Since $\phi_0$ is
$F$-linear, it follows that $\phi_0(dx) = d\phi_0(x) = 0$. Since $\phi_0(x) =
1_F$, it follows that $d=m$. 
\end{proof}
\begin{lem}\label{abelian}
Let $H$ be a finitely generated $F$-module. In case $F=\zm$ and hence $H$ is a
finite abelian group, then $H=\oplus_{i=1}^{r}A_i$, where $r=r(H)$ and
$A_i$'s are cyclic groups. Moreover given any element $x_1 \in H$ with order of $x_1$
being equal to the exponent of $H$, there exist $A_i$'s as above with $A_1 =
\la x_1\ra$.
\end{lem}
\begin{proof}
From the structure theorem of finite abelian groups, we have that
$H=\oplus_{i=1}^{s}A_i$, where $s$ is a positive integer and
$A_i$'s are cyclic groups isomorphic to $\Z/c_i\Z$ with $c_i|c_{i-1}$ for all
$2\leq i \leq s$. Moreover going through the proof of~\cite[Theorem
2.14.1]{Herstein} the last claim of the lemma follows. To prove the lemma
we need to show that $s=r$.  A subset of $H$ containing an element $x_i$
from each $A_i$ with $x_i$ being a generator of $A_i$, is of cardinality
$s$ and spans $H$ as an $F$-module. Thus from the definition of the rank of an
$F$-module we have
\begin{equation}\label{rless}
r\leq s.
\end{equation}
Moreover using the definition of a rank of an $F$-module we have a surjective
group homomorphism $f:\Z^r\to H$. Since $\Z^r$ is a free module over the
principle ideal domain $\Z$, we have that $\ker(f)$ is also a free module over
$\Z$. Moreover there exist a basis $\{y_1,\ldots,y_{r}\}$ of $\Z^r$ such
that the basis of $\ker(f)$ is $\{u_1y_1,\ldots,u_{r}y_{r}\}$, where
$u_i$'s are positive integers. Thus $ \Z^r/\ker(f) =
\oplus_{i=1}^{r}\Z/u_i\Z$. Since $H$ is isomorphic to $\Z^r/\ker(f)$ it
follows that $H$ can be written as a direct sum of $r$ cyclic groups.
But we also have that $H$ is isomorphic to $\oplus_{i=1}^{s}\Z/c_i\Z$ with $c_i|c_{i-1}$
for any $i$ which satisfies $2\leq i \leq s$. The condition that $c_i|c_{i-1}$ implies that $s$
is the least positive integer $d$ such that $H$ can be written as a direct sum
of $d$ cyclic groups. Therefore we have
\begin{equation}\label{sless}
s \leq r.
\end{equation}
Combining~\eqref{rless} and \eqref{sless} we have $s=r$. Hence the lemma is proven.
\end{proof}
\begin{lem}\label{X}
There exists a subset $X= \{x_1,\ldots,x_{r}\}$ of $\la A_{r,s}\ra$ of
 cardinality
 $r=r(\la A_{r,s}\ra)$ such that $x_1  \in A_{r,s}$ and $\la X\ra = \la A_{r,s}\ra$.
\end{lem}
\begin{proof}
In case $F$ is a field, we have a subset $X$ of $A_{r,s}$ such that $X$ forms
a basis of the vector space $\la A_{r,s}\ra$. Thus the claim follows in this
case. In case $F = \zm$, then from Lemma~\ref{exponent}, the order of any
element in~$A_{r,s}$ is equal to the exponent of $H$. Then using
Lemma~\ref{abelian} we have that $\la A_{r,s}\ra = \oplus_{i=1}^{r}A_i$ with
$A_i=\la x_i\ra$
and $x_1\in A_{r,s}$. Therefore $X= \{x_1,\ldots,x_{r}\}$ is a subset of
$\la A_{r,s}\ra$ satisfying the assertion of the lemma.
\end{proof}
\begin{prop}\label{mrp}
Let $A_{r,s} = \{\bar{e_1},\ldots,\bar{e_{k_1}}\}$ be as above. Then the
rank of the submodule 
$H_A = \la\bar{e_2}-\bar{e_1},\ldots,\bar{e_{k_1}}-\bar{e_1}\ra$ of $\la A_{r,s}\ra$ is
  equal to $r_s(A) = r(\la A_{r,s}\ra)-1.$
\end{prop}
\begin{proof}
Since $A_{r,s}$ is contained in $H_A + \bar{e_1}$ and from Lemma~\ref{rsetmod}
the rank of $\la A_{r,s}\ra$
is equal to $r_s(A) +1$, it follows that 
$r(H_A) \geq r_s(A).$ For proving the lemma we shall show that $H_A$ is
contained in a module $H$ of rank at most $r_s(A)$. Let $X = \{x_1,\ldots,x_{r}\}$
be a subset of $A_{r,s}$ with $x_1 = \bar{e_1}$ and  $r = r_s(A)+1$ as provided by
Lemma~\ref{X}. Since $\la X\ra = \la A_{r,s}\ra$, for any $i$ with $1\leq i \leq k_1$, there exists $\lambda_{j,i} \in F$ such that
\begin{equation}\label{contain}
\bar{e_i} = \sum_{j=1}^{r}\lambda_{j,i}x_j.
\end{equation} 
Let $\phi_0$ be the $F$-linear map as in Lemma~\ref{exponent}. Then
evaluating the value of  the both sides of the above equality for the map $\phi_0$, we obtain that 
$$1_F = \sum_{j=1}^{r}\lambda_{j,i}\phi_0(x_j).$$ 
Moreover since $x_1 = \bar{e_1}$ and thus $\phi_0(x_1) = \phi_0(\bar{e_1}) = 1_F$, it follows that for any $i$, we have $\lambda_{1,i} = 1 - \sum_{j=2}^{r}\phi_0(x_j)$. Using this and~\eqref{contain} it follows that
$A_{r,s} \subset x_1 + H$ where $H$ is the module $\la x_2-\phi_0(x_1)x_1,\ldots,x_{r}-
\phi_0(x_{r})x_1\ra$. Thus $H$ contains $H_A$ and its rank is clearly less than or equal to
$r-1$. Therefore it follows that $r(H_A) \leq r-1 = r_s(A)$. Hence the lemma follows.
\end{proof}
\begin{cor}\label{analogous}
Let $A$ be a finite subset of an $F$-module $H$. Then $r_s(A)$ is the largest integer $d$ such that
$A$ is Freiman $s$-isomorphic to a subset $X$ of a module $H$ of rank $d$ and $X$ is not
contained in a translate of any proper submodule of $H$.
\end{cor}
\begin{proof}
From Lemma~\ref{rsetmod}, we have $r_s(A) =r_s(A_{r,s})$. Let $B =
\{0,\bar{e_2}-\bar{e_1},\ldots,\bar{e_{k_1}}-\bar{e_1}\}$. Then we have a
Freiman $s$-isomorphism
$f:A_{r,s}\to B$ defined by $f(\bar{e_i}) = \bar{e_i}-\bar{e_1}$. From 
Proposition~\ref{mrp} the rank of the module $\la B\ra =
H_A$  is equal to $r_s(A)$. Moreover we observe that if $B$ is contained in
$H' +x$ for some submodule $H'$ of $H$, then since $B$ contains $0$, it
follows that $x \in H'$ and $H' = H_A= \la B\ra$. In other words $B$ is not contained
in a translate of any proper submodule of $\la B\ra$. This implies that
 $d \geq r_s(A)$. Now using Lemma~\ref{extension} any Freiman $s$-isomorphism
 $f:A_{r,s}\to X$ extends as a $F$-linear map  $\wt{f}:\la A_{r,s}\ra \to
 \la X\ra$. Since $A_{r,s} \subset H_A + \bar{e_1}$, we have that $X \subset
 \wt{f}(H_A) + \wt{f}(\bar{e_1})$. Since the rank of $\wt{f}(H_A)$ is at most the
 rank of $H_A$ which is equal to $r_s(A)$, it follows that any set isomorphic
 to $A$ is contained in a translate of a module of rank at most $r_s(A)$. This
 implies that $d \leq r_s(A)$. Hence $r_s(A) = d$.
\end{proof}
\section{Concluding remarks}
A subset $A$ of an abelian group $G$ is said to be {\em sum-free} if there
is no solution of the equation $x+y =z$ with $x,y,z \in
A$. In~\cite{BGS} it was shown that the problem of obtaining an upper
bound for the number of sum-free
sets in certain types of finite abelian groups is equivalent to obtaining an
upper bound for
\begin{equation}\label{ah}
a(H) = \sum_{k_1,k_2}\frac{\Card(S(k_1,k_2,H))}{2^{k_2}},
\end{equation}
with $H= G/(\Z/m\Z)$, where $m$ is the exponent of $G$.
Using the  upper bound for $\card(S(k_1,k_2,H))$ provided by
Theorem~\ref{smallG} it follows that 
\begin{equation}\label{ubch}
a(H) \leq n^{n^{2/3\log n}},
\end{equation}
where $n$ is the order of $H.$
One could also show that
\begin{equation}\label{lbch}
a(H) \geq \frac{s(H)}{2},
\end{equation}
where $s(H)$ is the number of subgroups of $H.$  
Using Theorem~\ref{smallG}, one may verify that the main contribution in the right hand side of~\eqref{ah} comes from those summands with 
 $(2-\epsilon)k_1 \leq k_2 \leq (2+\epsilon)k_1.$
\begin{center}
{\bf Acknowledgement}
\end{center}~I thank R. Balasubramanian, D.S. Ramana for many helpful discussions and carefully reading the manuscript. I would also like to thank Jean-Marc Deshouillers,
Imre Ruzsa and Gilles Z\'emor for making several useful comments. A part of this work was done when I was a post
doctoral fellow at Harish-Chandra research institute (HRI), Allahabad,
India. I am grateful for the support I received during my stay at HRI.

\begin{flushleft}
Institute of Mathematical Sciences,\\
CIT Campus, Taramani,\\
Chennai-600113,\\
India\\
gyan.jp@gmail.com
\end{flushleft}
\end{document}